\documentclass[11pt, reqno]{amsart}
\usepackage{hyperref}
\usepackage{cmap}                        
\usepackage[cp1251]{inputenc}            
\usepackage[english]{babel}
\usepackage[left=1in,right=1in,top=1in,bottom=1.5in]{geometry} 
\usepackage{amssymb,amsmath, amsthm, amscd,ifthen}
\usepackage{graphicx}
\usepackage{epigraph, xcolor, hyperref}

\newtheorem*{thm*}{Theorem}
\newcommand{\ff}{{\mathcal F}}
\newcommand{\aaa}{{\mathcal A}}

\newcommand{\G}{{\mathcal G}}

\newtheorem*{cla*}{Claim}

\newtheorem{thm}{Theorem}
\newtheorem{gypo}{Conjecture}
\newtheorem{opr}{Definition}
\newtheorem{lem}[thm]{Lemma}
\newtheorem{cla}[thm]{Claim}

\newtheorem*{prb}{Problem}

\newtheorem{cor}[thm]{Corollary}
\date{}

\newtheorem{prop}[thm]{Proposition}

\newtheorem{defn}[thm]{Definition}

\date{}
\title{Beyond the Erd\H os Matching Conjecture}
\author{Peter Frankl}\address{R\'enyi Institute, Budapest, Hungary; MIPT, Moscow, Russia; Email: {\tt peter.frankl@gmail.com}}

\author{Andrey Kupavskii}
\address{CNRS, Grenoble, France; IAS, Princeton, US;
MIPT, Moscow, Russia; Email: {\tt kupavskii@ya.ru}.}
\begin{document}
\maketitle
\begin{abstract}
  A family $\ff\subset {[n]\choose k}$ is $U(s,q)$ of for any $F_1,\ldots, F_s\in \ff$ we have $|F_1\cup\ldots\cup F_s|\le q$. This notion generalizes the property of a family to be $t$-intersecting and to have matching number smaller than $s$.

  In this paper, we find the maximum $|\ff|$ for $\ff$ that are $U(s,q)$, provided $n>C(s,q)k$ with moderate $C(s,q)$. In particular, we generalize the result of the first author on the Erd\H os Matching Conjecture and prove a generalization of the Erd\H os--Ko--Rado theorem, which states that for $n> s^2k$  the largest family $\ff\subset {[n]\choose k}$ with property $U(s,s(k-1)+1)$ is the star and is in particular intersecting. (Conversely, it is easy to see that any intersecting family in ${[n]\choose k}$ is $U(s,s(k-1)+1)$.)

  We investigate the case $k=3$ more thoroughly, showing that, unlike in the case of the Erd\H os Matching Conjecture, in general there may be $3$ extremal families.
\end{abstract}
\section{Introduction}

Let $[n] = \{1,\ldots, n\}$ be the standard $n$--element set and let ${[n]\choose k}$ denote the collection of all its $k$-subsets, $1\le k<n$. A {\it $k$-graph} (or a {\it $k$-uniform family}) $\ff$ is simply a subset of ${[n]\choose k}$. Let us recall two fundamental results from extremal set theory.

\begin{thm}[Erd\H os--Ko--Rado Theorem \cite{EKR}] Let $t$ be a positive integer, $t\le k$ and suppose that $|F\cap F'|\ge t$ for all pairs of edges of the $k$-graph $\ff$. Then
\begin{equation}\label{eq1} |\ff|\le {n-t\choose k-t}
\end{equation}
holds for all $n\ge n_0(k,t)$. Moreover, if $t=1$ then we can take $n_0(k,t) = 2k$.
\end{thm}
The family $\big\{F\in{[n]\choose k}: [t]\subset F\big\}$ shows that \eqref{eq1} is the best possible.

Let $p,r$ be non-negative integers with $p\ge r$. Define

\begin{equation}\label{defa}\aaa_{p,r} := \aaa(p,r,n,k):=\Big\{A\in {[n]\choose k}: |A\cap [p]|\ge r\Big\}.\end{equation}
(We omit $n,k$ when they are clear from the context.) The Erd\H os--Ko--Rado family mentioned above is simply $\aaa_{r,r}$. These families arise in several important results and open questions in extremal set theory.\vskip+0.1cm

The following extension of the Erd\H os--Ko--Rado theorem was conjectured by Frankl \cite{F1}, proved in many cases by Frankl and F\" uredi \cite{FF}, and nearly 20 years later proved in full generality by Ahlswede and Khachatrian \cite{AK2}.

\begin{thm}[Complete Intersection Theorem \cite{AK2}]\label{thmak} Suppose that $\mathcal F\subset{[n]\choose k}$ is $t$-intersecting, $n\ge 2k-t$. Then
\begin{equation}\label{art10_eqstar} |\mathcal F|\le \max_{0\le i\le k-t}|\mathcal A_{2i+t,i+t}|.
\end{equation}
Moreover, unless $n = 2k, t=1$ or $\mathcal F$ is isomorphic to $\mathcal A_{2i+t,i+t}$, the inequality is strict.
\end{thm}

For a $k$-graph $\ff$ let $\nu(\ff)$ denote its matching number, that is, the maximum number of pairwise disjoint edges in $\ff$. Obviously, for every positive integer $s$, $\nu\big({[(s+1)k-1]\choose k}\big)=s$.


The Erd\H os Matching Conjecture (the EMC for short) is one of the central open problems in extremal set theory.

\begin{gypo}[EMC \cite{E}] Let $n\ge (s+1)k$ and $\ff\subset {[n]\choose k}$. If $\nu(\ff)\le s$ then
\begin{equation}\label{eq2} |\ff|\le \max\big\{|\aaa_{s,1}|, |\aaa_{(s+1)k-1,k}|\big\}.
\end{equation}
\end{gypo}
We note that both families appearing on the right hand side have matching number $s$. It is known to be true for $k\le 3$ (cf. \cite{EG}, \cite{F11}).

We should mention that Erd\H os proved \eqref{eq2} for $n\ge n_0(k,s)$. For such values of $n$ the bound is \begin{equation}\label{eq3}
           |\ff|\le {n\choose k}-{n-s\choose k}.
         \end{equation}
Improving earlier bounds \cite{BDE,HLS}, \eqref{eq3} was proved by the first author in \cite{F4} for $n\ge (2s+1)k$. For $s\ge s_0$ this was further improved by the present authors to $n\ge \frac 53sk$ (cf. \cite{FK16}).

Both the above results forbid certain intersection patterns (two sets intersecting in less than $t$ elements or $s+1$ sets having pairwise empty intersection).
In the present paper, we study restrictions on the maximum size of the union, rather than intersections, of $s+1$ edges of $\ff$. This setting permits to unify the above two results and to formulate a natural new problem.

\begin{defn}
  Let $k,s\ge 2$ and $k\le q<sk$ be integers. A $k$-graph $\ff\subset {[n]\choose k}$ is said to have property $U(s,q)$ if
  \begin{equation}\label{eq4}
    |F_1\cup\ldots\cup F_s|\le q
  \end{equation}
  for all choices of $F_1,\ldots, F_s\in \ff$. For shorthand, we will also say $\ff$ is $U(s,q)$ to refer to this property.
\end{defn}
With this definition, $\ff$ being $U(2,2k-t)$ is equivalent to  $\ff$ being $t$-intersecting and, similarly, $\ff$ being $U(s+1,(s+1)k-1)$ is equivalent to $\nu (\ff)\le s$.

\begin{defn}
  Let $n,k,s,q$ be integers, $n>q\ge k$, $sk>q$, $s\ge 2$. Define $m(n,k,s,q)$ as the maximum of $|\ff|$ over all $\ff\subset {[n]\choose k}$, where $\ff$ has property $U(s,q)$.
\end{defn}
With this terminology, Theorem~\ref{thmak} states that
\begin{equation}\label{eq5}
  m(n,k,2,2k-t)=\max_{0\le i\le k-t}|\aaa_{2i+t,i+t}|,
\end{equation}
and the EMC can be formulated as
{\small $$m(n,k,s+1,k(s+1)-1) = \max\big\{|\aaa_{s,1}|, |\aaa_{(s+1)k-1,k}|\big\}\ \ \  \Big(=\max\Big\{{n\choose k}-{n-s\choose k}, {k(s+1)-1\choose k}\Big\}\Big).$$}

For all choices of $A_1,\ldots, A_s\in \aaa_{p,r}$, we have
\begin{equation}\label{eq6}
  |A_1\cup\ldots \cup A_s|\le p+s(k-r).
\end{equation}
Thus, the EMC and the Complete Intersection Theorem may be seen as particular cases of the following general conjecture.

\begin{gypo}\label{conj2}
  For all choices $n,k,s,q$ one has $m(n,k,s,q) = |\aaa_{p,r}|$ for an appropriate choice of $p, r> 0$ with $\aaa_{p,r}$ having property $U(s,q)$. More precisely, if $q = (k-r)s+p$ with $r\le p\le s+r-2$, then $m(n,k,s,q) = \max_{0\le i\le k-r} |\aaa_{p+is,r+i}|$.
\end{gypo}
In particular, Theorem~\ref{thmak} is the case $p=r=t$, $s=2$ of the conjecture.

Let us remark that the {\it non-uniform} version of Conjecture~\ref{conj2} goes back to the PhD dissertation of the first author (cf. also \cite{F76} and \cite{F79} where it is proved in a certain range). Let us also mention that the $s=2$ case of the non-uniform case is a classical result of Katona \cite{Ka1964}.\vskip+0.1cm

\subsection{Simple properties of $m(n,k,s,q)$} Let us give some additional motivation for the question. Recall the following definition.

\begin{opr}\label{defsh}
  Consider two sets $F_i = (a_1^i,\ldots, a_k^i)$ with $a_1^i< a_2^i<\ldots<a_k^i$ for $i=1,2$.
  Then $F_1\prec_s F_2$ iff $a_j^1\le a_j^2$ for every $j\in [k]$. We say that a family $\ff\subset {[n]\choose k}$ is
  \emph{shifted} if $F\in \ff$ and $G\prec_s F$ implies $G\in \ff$.
\end{opr}

For many extremal problems, including the ones mentioned above, it is sufficient to prove the statements for shifted families.
 Let $\ff\subset {[n]\choose k}$ be a shifted family satisfying $\nu(\ff)\le s$. In dealing with such families one often considers subfamilies of the form $\ff\big(\bar{[r]}\big):=\{F\in\ff:F\cap [r] = \emptyset\}$. As we noted above, $\nu(\ff)< s$ is equivalent to $\ff$ being $U(s,sk-1)$. Due to shiftedness, $\ff\big(\bar{[r]}\big)$ has property $U(s,ks-1-r)$. \vskip+0.1cm

There is a certain hierarchy for properties $U(s,ks-r)$ in the range $1\le r<s$.
To explain it, recall that $\aaa_{p,1}$ has property $U(s,(k-1)s+p)$ for $1\le p<s$.
\begin{prop}\label{prophier}
  If $m(n,k,s,(k-1)s+p)) = {n\choose k}-{n-p\choose k}$ then
  $$m(n+1,k,s,(k-1)s+p+1) = {n+1\choose k}-{n-p\choose k}.$$
\end{prop}
\begin{proof}
  Let $\ff\subset {[n+1]\choose k}$ be a shifted family satisfying property $U(s,(k-1)s+p+1)$. Then $\ff(\bar 1)$ is $U(s,(k-1)s+p)$ by shiftedness, and so $|\ff(\bar 1)|\le {n\choose k}-{n-p\choose k}$ follows. Since $|\ff(1)|\le {n\choose k-1}$, we get $|\ff|=|\ff(1)|+|\ff(\bar 1)|\le {n+1\choose k}-{n-p\choose k}$.
\end{proof}


Shiftedness also allows us to prove the following proposition.

\begin{prop}\label{propshiftedunion}
  Assume that $\ff$ is shifted and has property $U(s,q)$ for $q=(k-r)s+p$ with $0<r\le p\le s+r-2$. Then $\ff\subset \cup_{i=0}^{k-r}\aaa_{p+is,r+i}$. Consequently, $m(n,k,s,q)\le \sum_{i=0}^{k-r}|\aaa_{p+is,r+i}|$.
\end{prop}
\begin{proof} Assume that this is not the case. This means that there exists $A\in \ff$ such that $|A\cap [p+is]|\le r+i-1$ for $i = 0,\ldots, k-r$. Any such $A$  satisfies $A\succ_{s} B$, where $B:= \{1,\ldots, r-1,p+1,p+s+1,\ldots, p+(k-r)s+1\}$, and thus $B\in \ff$ because $\ff$ is shifted. But then $B_i:= \{1,\ldots, r-1,p+1-\min\{i,p-r+1\}, p+s+1-i,\ldots, p+(k-r)s+1-i\}$ also belong to $\ff$ for each $i=0,\ldots, s-1$. However, $\cup_{i=0}^{s-1} B_i = [(k-r)s+p+1] = [q+1]$, a contradiction with the $U(s,q)$ property.
\end{proof}


\subsection{Results}
In this paper, we start exploring the quantity $m(n,k,s,q)$. One of the most appealing cases of our general results is the following generalization of the ($t=1$ case of the) Erd\H os--Ko--Rado theorem.

\begin{thm}\label{thm1ekr} Fix some positive integers $n,k,s$. Assume that $n>s(s+2)k$ and $\ff\subset {[n]\choose k}$ has property $U(s+1,k+s(k-1))$. In words, the latter condition means that $|F_0\cup\ldots\cup F_{s}|\le k+s(k-1)$ for any $F_0,\ldots, F_s\in \ff$. Then $|\ff|\le {n-1\choose k-1}$.
\end{thm}
We note that any intersecting family $\ff\subset {[n]\choose k}$ has property  $U(s+1,(s+1)(k-1)+1)$ since for any $F, F'\in \ff$, $|F'\setminus F|\le k-1$, and so, for any $F_0,\ldots, F_s$, $|F_0\cup\ldots \cup F_s|\le |F_0|+|F_1\setminus F_0|+\ldots+|F_s\setminus F_0|\le k+s(k-1)$. Moreover, it is not difficult to see that, as long as $\ff$ is not a star and $s\ge k$, $\ff$ has property $U(s+1,q)$ for some $q<k+s(k-1)$.

Theorem~\ref{thm1ekr} is an immediate corollary of the  following general theorem. In it, we determine  $m(n,k,s+1,q)$ for all $q$ and $n>C(s) k$. Naturally, we are only interested in the values $q>k$.

\begin{thm}\label{thmsunion}
 Fix some integers $n,k,s,p,r$, such that $1\le r\le k$ and $r\le p\le s+r-1$. Suppose that $\ff\subset {[n]\choose k}$ has property $U(s+1,q)$ for $q=(k-r)(s+1)+p$.  If $n\ge p+1+(s+f(s,p,r))(k-r),$ \footnote{One may simply say $n>C(s)k$ instead, if the exact form of the dependence is not important} where
  $$f(s,p,r):=s(s+1) \cdot \frac{\sum_{j=0}^{r-1} s^{r-1-j}{p\choose j}}{{p\choose r}},$$
  then $$|\ff|\le |\aaa_{p,r}|.$$
\end{thm}

In particular, for $r=1$ and $p=s$ we retrieve the bound on the size of the family as in the Erd\H os Matching Conjecture, while for $r=t$ and $p=0$ we get the bound from the Erd\H os--Ko--Rado theorem.
 Note that  $|\aaa_{p,r}|=\sum_{i=r}^k{p\choose i}{n-p\choose k-i}\sim{p\choose r}{n-p\choose k-r}$ as $n\to \infty$.\vskip+0.1cm

{\bf Remarks. } The function $f(s,p,r)$ looks complicated, which is partially due to the generality of the statement. Let us illustrate it on a few examples. First, if one substitutes $r=1$ and $p=s$, then $f(s,p,r) = s(s+1)\frac 1{s} = s+1$, and we get the bound $n\ge s+1+(2s+1)(k-1)$, exactly the restriction in \cite[Theorem~1.1]{F4}, which guarantees that the Erd\H os Matching Conjecture holds in this range. For $r=1$ and $1\le p\le s$, we get that $f(s,p,r) = s(s+1)/p$, and, roughly speaking, Theorem~\ref{thmsunion} holds for $n\ge 2s^2k/p$.

For $r=p=1$, the theorem implies Theorem~\ref{thm1ekr}. \vskip+0.1cm

{\bf Sharpness. } The aforementioned rough bound $n\ge 2s^2k/p$ is sharp up to a small constant factor: it is not difficult to see that $\mathcal A_{s+p,2}$ is bigger than $\mathcal A_{p,1}$ for $n<cs^2k/p$ with, say, $c=1/4$. This, in particular, tells us that the bound in Theorem~\ref{thm1ekr} is just a constant factor away from the best possible. \vskip+0.1cm

{\bf $\mathbf{k=3}.$ } As we have already mentioned, the EMC is proved for $k=3$ by the first author. In the next section, we obtain some rather strong partial results concerning $m(n,3,s,q)$ for different values of $q$. Probably the most interesting feature of these results is that, depending on the value of $n = n(s)$, we will show that each of the three constructions $\aaa_{p,r}$ available for $k=3$ may be extremal. For precise results, see Theorems~\ref{thmsmallq},~\ref{thmf3},~\ref{thmf2}.\vskip+0.1cm

 {\bf Structure of the paper. } The proof of Theorem~\ref{thmsunion} is given in Section~\ref{secmain}. In the next section, we resolve the problem for some small values of $k,s,q$. Most of the results are devoted to the case $k=3$.

 \section{Results for small $k,s,q$}
\subsection{The complete solution for $k=2$.}
Throughout this section $\ff\subset {[n]\choose 2}$ is a $2$-graph having property $U(s,s+r)$ with $1\le r<s,n>s+r$.
Recall the definition \eqref{defa}. Then $\aaa(r,1,n,2)$ and $\aaa(s+r,2,n,2)$ are $U(s,s+r)$.
\begin{thm} For all values of $n,s,r,$ $n>s+r$, $s>r\ge 1$
\begin{equation}\label{eq41}
  |\ff|\le \max\big\{|\aaa(r,1,n,2)|,|\aaa(s+r,2,n,2)|\big\}.
\end{equation}
\end{thm}
The case $r=s-1$ of the theorem is a classical result of Erd\H os and Gallai, determining the maximum number of edges in a graph without $s$ pairwise disjoint edges.
\begin{proof}
Let $\cup \ff:=\cup_{F\in\ff} F$. In the case $|\cup \ff|\le s+r$, we have $|\ff|\le {s+r\choose 2}$. From now on we suppose that $|\cup \ff|>s+r$.
\begin{cla}
  $\nu(\ff)\le r$.
\end{cla}
\begin{proof}
  Indeed, if we can find $F_1,\ldots, F_{r+1}$ with $|F_1\cup \ldots \cup F_{r+1}| = 2(r+1)$, then $|\cup \ff|\ge s+r+1$ enables us to find $s-r-1$ edges $F_{r+2},\ldots, F_s$ such that $|F_1\cup\ldots \cup F_s|\ge |F_1\cup \ldots\cup F_{r+1}|+s-r-1 = s+r+1$. This contradicts $U(s,s+r)$.
\end{proof}
Applying the Erd\H os--Gallai theorem \cite{EG} (for a short proof, see \cite{AF} or the next section) to $\ff$ yields $$|\ff|\le \max\Big\{{2r+1\choose 2},{n\choose 2}-{n-r\choose 2}\Big\}$$ and proves \eqref{eq41}.
\end{proof}
\subsection{The case $q\le 2s+k-3$} Let us first state a very simple result.
\begin{cla}
 $ m(n,k,s,q) = {q\choose k}$ for $k\le q\le k+s-2$.
\end{cla}
\begin{proof}
  Let $\ff\subset {[n]\choose k}$ have property $U(s,q)$. Consider $X:=\cup_{F\in\ff}F.$
  If $|X|>q$ holds then we can easily find $F_1,\ldots, F_s$ satisfying $|F_1\cup \ldots \cup F_s|\ge \min\{k+s-1,|X|\}>q$, a contradiction. Thus $|\ff|\le {|X|\choose k}\le {q\choose k},$ as desired.
\end{proof}

\begin{thm}\label{thmsmallq}
  Let $\ff\subset {[n]\choose k}$ be $U(s,s+t+k-3)$ for $t\in[2,s]$. Then
  $$|\ff|\le \max\big\{\aaa(s+t+k-3,k,n,k),\aaa(t+k-3,k-1,n,k)\big\}.$$
\end{thm}
Let us note that this theorem for $k=2$ includes the aforementioned Erd\H os-Gallai result, as well as the result from the previous subsection. The proof is a generalization of the proof in \cite{AF}.
\begin{proof}
  W.l.o.g. assume that $\ff$ is shifted. Consider the collection of sets $\big\{(1,\ldots, k-2,i+k-2,2t+1-i+k-2): i\in [t]\big\}\cup\big\{(1,\ldots, k-1,j): j\in [2t+k-1, s+t+k-2]\big\}$. These are $t+(s-t)=s$ sets in total. Moreover, their union is $[s+t+k-2]$, and, therefore, one of these sets is not in $\ff$.
\begin{itemize}\item Since $\ff\not\subset\aaa(s+t+k-3,k,n,k)$ we get  $(1,\ldots, k-1,s+t+k-2)\in\ff$, which implies that all sets of the form $(1,\ldots,k-1,t')$, $k\le t'\le s+t+k-2$ are in $\ff$;
\item since $\ff\not \subset \aaa(t+k-3,k-1,n,k)$ we get $(1,\ldots, k-2, t+k-2,t+k-1)\in \ff$.
\end{itemize}
Therefore, one of the sets $(1,\ldots, k-2,\ell+k-2,2t+1-\ell+k-2)$ for some $\ell\in [2,t-1]$ is missing.
This implies that $\ff\subset {[2t-\ell+k-2]\choose k}\cup {[\ell+k-3]\choose k-1}\times {[2t-\ell+k-1,n]\choose 1}.$ Thus
$$|\ff|\le {2t-\ell+k-2\choose k}+{\ell+k-3 \choose k-1}(n-2t+\ell-k+2)=:f(\ell).$$
We note that $f(t) = |\aaa(t+k-3,k-1,n,k)|$, while $f(1)={2t+k-3\choose k}\le |\aaa(s+t+k-3,k,n,k)|$.
Therefore, to conclude the proof of the theorem, it suffices to show that $f(\ell)$ is an integer convex function for $\ell\in[t]$. (Recall that $g(x)$ is an {\it integer convex function} on a certain interval if $2g(x)\le g(x-1)+g(x+1)$ on that interval.) Note that if a function on an interval is convex then it is integer convex on the same interval.

Let us note that $g'(-x) = -g'(x)$, yielding $g''(-x) = g''(x)$.
 The function ${x\choose k}$ is convex for $x\ge k$, and thus the first term ${2t-\ell+k-2\choose k}$ is convex for $\ell\in [t]$. The second term $f_1(\ell)$ in the definition of $f(\ell)$ is integer convex. Indeed, for any $\ell \in [t]$, we have $f_1(\ell-1)+f_1(\ell+1)\ge \big({\ell-1+k-3\choose k-1}+{\ell+1+k-3\choose k-1}\big)(n-2t+\ell-k+2)$, which is bigger than $2{\ell+k-3\choose k-1}(n-2t+\ell-k+2)$ due to integer convexity of ${x\choose k-1}$ for $x\in \mathbb Z$.
 We conclude that $f(\ell)$ is integer convex for $\ell \in [t]$.
 \end{proof}

\subsection{The case $k=3$, $q=2s+1$}\label{sec2s+1} For $k=2$, (each of) the theorems from the previous two sections resolve the problem completely. For $k=3$, Theorem~\ref{thmsmallq} covers the cases with $q\le 2s$. Thus, the case mentioned in the title of this section is the ``first'' case, not covered by Theorem~\ref{thmsmallq}. As we would see, the situation changes quite significantly: instead of having two potential extremal families, we shall have three.

Consider a family $\ff\subset {[n]\choose 3}$ that is $U(s,2s+1)$. There are three natural candidates for extremal families here (cf. \eqref{defa}):

$$\ff_1:=\aaa(1,1,n,3), \ \ \ \ff_2:=\aaa(s+1,2,n,3), \ \ \ \ff_3:=\aaa(2s+1,3,n,3).$$
In particular, $\ff_1$ contains all sets that contain $1$ and $\ff_3:={[2s+1]\choose 3}$. It is easy to see that all three families are $U(s,2s+1)$. Let us make the following conjecture.

\begin{gypo}\label{conj3}
  If $\ff\subset {[n]\choose 3}$ is $U(s,2s+1)$ then $|\ff|\le \max_{i\in[3]}|\ff_i|$.
\end{gypo}
Unlike in the case of the Erd\H os Matching Conjecture, each of the three families is the largest for each $s\ge 3$ in a certain interval depending on $n$. Let us show that. We have
$$|\ff_1\setminus \ff_2| = {n-s-1\choose 2},\ \ \ \ |\ff_2\setminus \ff_1| = {s\choose 2}{n-s-1\choose 1}+{s\choose 3},$$
therefore, $|\ff_1|\ge |\ff_2|$ for roughly $n\ge s^2$, and smaller otherwise. Next, we have
$$|\ff_2\setminus \ff_3| = {s+1\choose 2}{n-2s-1\choose 1}, \ \ \ \ |\ff_3\setminus \ff_2| = {s\choose 3}+{s\choose 2}{s+1\choose 1},$$
and $|\ff_3|\ge|\ff_2|$ iff $3(s+1)(n-3s)\le (s-1)(s-2)$. For large $s$, this happens roughly for $n\le \frac {10}3 s$.
\begin{thm}\label{thmf3} Assume that $\ff\subset {[n]\choose 3}$ is $U(s,2s+1)$ and, moreover, $n\le 3s$ and $s\ge 10$. Then $$|\ff|\le |\ff_3|.$$
\end{thm}
Let us also note that Theorem~\ref{thmsunion} gives $|\ff|\le |\ff_1|$ for $n\ge 2s^2+4s+2$. In Section~\ref{secf2}, we will show (in a more general setting) that in a certain range $|\ff|\le |\ff_2|$ for any $\ff$ that is $U(s,2s+1)$.

\begin{proof} Let $\ff\subset {[n]\choose 3}$ be shifted and $U(s,2s+1)$, $|\ff|>\max |\ff_i|$, $i=1,2,3$. In particular, $(1,2,2s+2)\in \ff$, $(2,3,4)\in \ff$.

Consider , $\G_2:=\ff_2\setminus \ff_3$, $\G_1:=\ff_1\setminus(\ff_2\cup \ff_3)$.  Then
\begin{align*}
  \G_2 =& \{(a,b,c): b\le s+1, 2s+1<c\le n\},\\
  \G_1 =& \{(1,b,c): b\ge s+2,c\ge 2s+2\}.
\end{align*}
Recall that Proposition~\ref{propshiftedunion} implies that
$\ff\subset \ff_1\cup \ff_2\cup \ff_3$ and thus also
$$\ff\subset \G_1\cup \G_2\cup \ff_3.$$

\begin{cla}\label{cla11}
  $(1,s+3,2s+2)\notin\ff$. Consequently, $|\G_1\cap \ff|\le n-2s-1$.
\end{cla}
\begin{proof}
  Otherwise, $(1,s+3-\ell,2s+2-\ell)\in \ff$ for $\ell = 0,1,\ldots, s-2$. Together with $(2,3,4)$ these are $s$ sets with  union $[2s+2]$. As for the second part, only sets from $\G_1$ with $b=s+2$ may be included in $\ff$.
\end{proof}

{\bf Remark.} Should $(2,3,s+3)\in \ff$, $\G_1\cap \ff = \emptyset$ would follow in a similar way.

\begin{cla}
  We have $\nu(\ff)\le \frac {s+1}2$.
\end{cla}
\begin{proof}
  Indeed, otherwise there are $\lceil s/2\rceil+1$ sets in $\ff$, whose union is $[\lceil 3s/2\rceil +3]$. Together with $\lfloor s/2\rfloor -1$ sets $(1,2,\lceil 3s/2\rceil +4),\ldots, (1,2,2s+2)$, we get $s$ sets whose union is $[2s+2]$.
\end{proof}

Since the Erd\H os Matching Conjecture holds for $k=3$ (cf. \cite{F11}), we have $|\ff_3\cap \ff|\le {2s+1\choose 3}-{2s+1-\frac {s+1}2\choose 3}$.

To prove $|\ff|\le |\ff_3|$ we need to show
\begin{equation}\label{ineq1}
  {2s+1-\frac{s+1}2\choose 3}>|\G_2|+|\G_1\cap \ff|.
\end{equation}
We have $|\G_2|\le {s+1\choose 2}(n-2s-1)$. Using Claim~\ref{cla11}, we get that the right hand side of \eqref{ineq1} is at most
$$\Big({s+1\choose 2}+1)\Big)(n-2s-1).$$
On the other hand, the left hand side of \eqref{ineq1} is
$$\Big(\frac 9{16}s^2-\frac 1{16}\Big)(s-1).$$
Given that $n\le 3s$, the left hand side of \eqref{ineq1} is bigger than the right hand side if
$${s+1\choose 2}+1\le \frac 9{16}s^2-\frac 1{16},$$
which holds for any $s\ge 10$.
\end{proof}

 \subsection{The complete solution for $k=s=3$, $q=7$.}
 In this section, we prove a stronger form of Conjecture~\ref{conj3} for these parameters.
 \begin{thm}\label{thmu37}
   Let $\ff\subset {[n]\choose 3}$ be shifted and satisfy $U(3,7)$. Then  \begin{equation}\label{equ37}\ff< \max_{i\in[3]}|\ff_i|,\end{equation} unless $\ff\in \{\ff_i:i\in[3]\}.$
 \end{thm}
Simple computation shows that $\max_{i\in[3]}|\ff_i|$ is given by $|\ff_3|$ for $n\le 9$, $|\ff_2|$ for $10\le n\le 11$ and $|\ff_1|$ for $n\ge 12$.
\begin{proof}[Proof of Theorem~\ref{thmu37}]
The statement is obvious for $n\le 7$ since $\ff_3 = {[7]\choose 3}$ is $U(3,7)$. In what follows, we assume that $n\ge 8$.
Arguing indirectly, assume that $\ff$ is not contained in any of $\ff_i$. Then, by shiftedness, this implies that
\begin{itemize}
  \item[(i)] $(1,2,8)\in\ff;$
  \item[(ii)] $(1,5,6)\in \ff;$
  \item[(iii)] $(2,3,4)\in \ff.$
\end{itemize}
\begin{cla}
  $(1,6,8)\not\in\ff$ and $(2,4,8)\notin\ff$.
\end{cla}
\begin{proof}
  Using shiftedness, if $(1,6,8)\in\ff$ then $(1,5,7)\in\ff$. Together with (iii) this contradicts $U(3,7)$.

  Similarly, if $(2,4,8)\in\ff$ then $(1,3,7)\in \ff$ and (ii) gives the contradiction with $U(3,7)$.
\end{proof}

\begin{cla}\label{cla121}
  $\big|\ff\cap{[7]\choose 3}\big|\le {7\choose 3}-10 = 25$.
\end{cla}
\begin{proof}
  There are $10$ pairs $F,F'\in {[7]\choose 3}$ with $F\cup F' = [2,7]$, and both sets in each pair cannot appear together in $\ff$ due to (i).
\end{proof}

\begin{cla}
  $(1,4,7)\in\ff$.
\end{cla}
\begin{proof}
  The contrary would imply $(a,b)\in {[3]\choose 2}$ for $(a,b,c)\in\ff$ with $c\ge 7$. Consequently, $|\ff|\le {6\choose 3}+3(n-6)$.

  However, the right hand side of \eqref{equ37} is at least $|\ff_2| = 4+6(n-4)=16+6(n-6)>{6\choose 3}+3(n-6)$ for $n\ge 8$.
\end{proof}
Combined with (ii), we get the following corollary.
\begin{cor}
  $(2,3,8)\notin \ff$.
\end{cor}
Now $(1,6,8)\notin \ff$ implies
\begin{equation}\label{eq35}|\ff(n)|\le 4 \ \ \ \ \text{for all }n\ge 8.
\end{equation}
Using Claim~\ref{cla121}, we infer that
$$|\ff|\le 25+4(n-7).$$
The right hand side is strictly less than $|\ff_3|=35$ for $n=8,9$ and strictly less than $|\ff_2| = 22+6(n-7)$ for $n\ge 10$. This concludes the proof.
\end{proof}
\subsection{The case $k=3$, $q=2s+t$ for small $t$}\label{secf2} In this section, we study families $\ff\subset {[n]\choose 3}$ that are $U(s,2s+t)$ for $t\le \epsilon s$ with some small $\epsilon>0$. We show in particular that $\ff_2$ from Section~\ref{sec2s+1} is the largest in a certain range.
Reusing the notation from Section~\ref{sec2s+1}, let us define the following families:
$$\ff_1:=\aaa(t,1,n,3), \ \ \ \ff_2:=\aaa(s+t,2,n,3), \ \ \ \ff_3:=\aaa(2s+t,3,n,3).$$
\begin{thm}\label{thmf2}
  Let $\ff \subset {[n]\choose 3}$ be $U(s,2s+t)$ with $s>t\ge 1$. Assume that $5(s+t)\le n\le \frac{(s+t)^2}{3t}$.
  Then $|\ff|\le |\ff_2|$.
\end{thm}
As we have seen in Theorems~\ref{thmsunion} and~\ref{thmf3}, both the upper and the lower bounds on $n$ are tight up to constants.
\begin{proof}
  Take $\ff$ as in the theorem and w.l.o.g. assume that $\ff$ is shifted and $\ff\not\subset \ff_2$. This implies $(1,s+t+1,s+t+2)\in \ff$.

  Let us define the following $t+1$ sets $B_1,\ldots, B_{t+1}$:
  $$B_i:=(i,2t+3-i,2s+t+2-i), \ \ \ i=1,\ldots, t+1.$$
  Clearly, $$B_1\cup\ldots \cup B_{t+1} = [2t+2]\cup [2s+1,2s+t+1].$$
  Next, define $A_1,\ldots, A_{s-t-1}$ by
  $$A_i:=(1,2t+2+i, 2s+1-i),\ \ \ i=1,\ldots, s-t-1.$$
  Then $A_{s-t-1}= (1,s+t+1,s+t+2)\in \ff$ and
  $$A_1\cup\ldots\cup A_{s-t-1} = \{1\}\cup [2t+3,2s].$$
  In particular, the union of the $s$ sets $B_1,\ldots, B_{t+1}$ and $A_1,\ldots, A_{s-t-1}$ is $[2s+t+1]$. Consequently, at least one of them is missing from $\ff$.

  We prove the theorem separately according to whether $B_i\notin \ff$ for some $i$ or $A_j\notin \ff$ for some $j$.

\begin{prop}
  If not all $B_i$, $1\le i\le t+1$, are in $\ff$ then $\nu(\ff\setminus \ff_3)\le t$.
\end{prop}
  \begin{proof}
    Note first that $F\in\ff\setminus \ff_3$ iff $F = (a,b,c)$ with $c\ge 2s+t+1$. It is easy to see that $\nu(\ff\setminus \ff_3)\ge t+1$ implies that there are sets $F_1,\ldots, F_{t+1}\in\ff\setminus \ff_3$ such that $F_1\cup\ldots\cup F_{t+1} = [2t+2]\cup\{2s+t+1\}$.

    It is not difficult to see that, for any $B_i$, $i\in[t+1]$, there is $F_j$ such that $B_i\prec_s F_j$, and thus $B_i\in\ff$, a contradiction.
  \end{proof}
Since $n\ge 5(s+t)>6t$, we have $|\ff\setminus\ff_3|\le {n\choose 3}-{n-t\choose 3}$ and $|\ff|\le {n\choose 3}-{n-t\choose 3}+{2s+t\choose 3}$.
On the other hand, we have $|\ff_2| = {s+t\choose 3}+{s+t\choose 2}(n-s-t).$
An easy, but tedious, calculation shows that the first inequality below holds. $$|\ff_2|-|\ff| \ge \frac {(s+t)^3}6+\frac {(s+t)^2}2(n-s-t)-t{n\choose 2}-\frac{(2s+t)^3}6\ge \frac n2((s+t)^2-tn)-\frac 53(s+t)^3.$$
Since $5(s+t)\le n\le \frac{(s+t)^2}{3t}$, the right hand side is at least $\frac 52(s+t)\frac{2(s+t)^2}3-\frac 53(s+t)^3\ge 0$.

Thus, we may assume that $A_{\ell} = (1,2t+2+\ell, 2s+1-\ell),$ $\ell \in [s-t-2]$, is not in $\ff$. We want to conclude the proof by showing that $|\ff_2\setminus \ff|\ge |\ff\setminus \ff_2|$.

For any $(a,b,c)\in\ff\setminus \ff_2$, we have $s+t<b<c\le 2s-\ell$. Thus, we can bound $$|\ff\setminus \ff_2|\le {s-t-\ell\choose 3}+(s+t){s-t-\ell\choose 2}.$$
On the other hand, $(a,b,c)\in\ff_2\setminus \ff$ if $2t+2+\ell\le b\le s+t$, $2s+1-\ell\le c\le n$. The number of choices of $c$ is at most $n-2s$, and the number of choices of $(a,b)$ is ${s+t\choose 2}-{2t+1+\ell\choose 2}.$
Thus, to show that $|\ff\setminus \ff_2|\le |\ff_2\setminus \ff|$, we need to show
$${s-t-\ell\choose 3}+(s+t){s-t-\ell\choose 2}\le (n-2s)\Big({s+t\choose 2}-{2t+1+\ell\choose 2}\Big).$$
Let us show this for some range of $\ell$ by (reverse) induction on $\ell$. For $\ell = s-t-2$, the left hand side is $s+t$, and the right hand side is $(n-2s)(s+t-1),$ and thus the inequality is satisfied. When passing from $\ell$ to $\ell-1$, the LHS increases by ${s-t-\ell\choose 2}+(s+t)(s-t-\ell) \le (s-t-\ell)\frac{3s+t-\ell}2$, and the RHS increases by $(n-2s)(2t+1+\ell)$. Since $n\ge 5(s+t)$, we are good as long as $2t+1+\ell\ge (s-t)/2$. If this inequality does not hold, then
the RHS of the displayed inequality is at least          $(n-2s)\cdot \frac 34{s+t\choose 2}$, which is bigger than the LHS as long as $\frac 34(n-2s)\ge \frac s3+s=\frac 43s$. This holds for $n\ge 5(s+t)\ge 5s.$

\end{proof}

\section{Proof of Theorem~\ref{thmsunion}}\label{secmain}
We say that the families $\ff_1,\ldots,\ff_{s+1}$ are {\it cross-dependent} if there are no $F_1\in\ff_1,\ldots, F_{s+1}\in \ff_{s+1}$ such that $F_i\cap F_j=\emptyset $ for any $i\ne j$. We say that the families $\ff_1,\ldots, \ff_s$ are {\it nested} if $\ff_1\supset\ldots\supset\ff_s$.\\

Since the $U$-property is preserved by shifting, we may  w.l.o.g.  assume that $\ff$ is shifted.


Denote $\mathcal G(B):=\big\{G\setminus [p+1]:G\in \mathcal G, G\cap [p+1]=B\big\}.$
Then \begin{equation}\label{eq001}
|\ff(B)|\le |\aaa_{p,r}(B)|\ \ \ \text{ holds for all }\ \ \ B\subset [p+1],\ |B|\ge r+1.
     \end{equation}

\begin{cla}\label{cla0}
  For any $B\subset [p+1]$, $|B|\le r-1$, we have $\nu(\ff(B))\le s$.
\end{cla}
\begin{proof}
  Indeed, the opposite gives $s+1$ members of $\ff$ whose union has size $|B|+(k-|B|)(s+1)\ge r-1+(k-r+1)(s+1)=s+r+(k-r)(s+1)>p+(k-r)(s+1)$, a contradiction with the $U(s+1,q)$ property.
\end{proof}
Similarly, we can obtain the following claim.
\begin{cla}\label{cla1} For fixed $B\in {[p+1]\choose r-1}$, any $s+1$ families $\ff(\{i_1\}\cup B),\ldots, \ff(\{i_{s+1}\}\cup B)$, where $i_1,\ldots, i_{s+1}\in [p+1]\setminus B$ and $\bigcup_{j\in [s+1]}\{i_j\} = [p+1]\setminus B$ (note that some of the $i_j$ will coincide for $p<s+r$), are cross-dependent.
\end{cla}
\begin{proof}
  Indeed, the opposite gives $s+1$ members of $\ff$ whose union has size $p+1+(k-r)(s+1)$, a contradiction.
\end{proof}

Let us recall that the following theorem was proved in \cite{F4}.
\begin{thm}[\cite{F4}]\label{thmshadowmatch} If $\mathcal G\subset {[m]\choose l}$ satisfies $\nu(\mathcal G)\le s$ then $s|\partial \mathcal G|\ge |\mathcal G|$.
\end{thm}

Note that, for $G$ and $G'$ of the same size and such that $G\prec_s G'$ (cf. Definition~\ref{defsh}), we have $\ff(G)\supset \ff(G')$ due to the fact that $\ff$ is shifted. Similarly, $$\partial \ff(B)\subset \ff(B\cup\{i\})\ \ \text{ for any }\ \ B\subset [p+1]\text{ and }i\in [p+1]\setminus B,$$
and, combined with Theorem~\ref{thmshadowmatch}, we get that
\begin{equation}\label{eqprojup}|\ff(B)|\le s|\ff(B\cup\{i\})|\ \ \text{ for any }\ \ B\subset [p+1]\text{ and }i\in [p+1]\setminus B.\end{equation}
Applying \eqref{eqprojup} to any given $B\in{[p]\choose i}$, we can get that, for any $0\le j<i$, $$\sum_{X\in {B\choose j}}|\ff(X)|\le s^{i-j}{i\choose j}|\ff(B)|.$$
Summing this over all $B\in{[p]\choose i}$ and noting that each $j$-subset of $[p]$ is contained in ${p-j\choose i-j}$ $i$-subsets of $[p]$, we get that
$${p-j\choose i-j}\sum_{X\in {[p]\choose j}}|\ff(X)|=\sum_{B\in {[p]\choose i}}\sum_{X\in {B\choose j}}|\ff(X)|\le s^{i-j}{i\choose j}\sum_{B\in {[p]\choose i}}|\ff(B)|.$$
Using that ${i\choose j}/{p-j\choose i-j} = {p\choose j}/{p\choose i}$, we get that, for any $i\ge 0$,
$$\sum_{B\subset [p], |B|\le i}|\ff(B)|\le \sum_{j=0}^{i}\frac{s^{i-j}{p\choose j}}{{p\choose i}}\sum_{B\in {[p]\choose i}}|\ff(B)|\le s\sum_{j=0}^{i}\frac{s^{i-j}{p\choose j}}{{p\choose i}}\sum_{B\in {[p]\choose i}}\big|\ff(B\cup\{p+1\})\big|.$$
Similarly,
$$\sum_{B\subset [p], |B|\le i}|\ff(B\cup\{p+1\})|\le \sum_{j=0}^{i}\frac{s^{i-j}{p\choose j}}{{p\choose i}}\sum_{B\in {[p]\choose i}}|\ff(B\cup\{p+1\})|.$$
Summing these two expressions for $i=r-1$, we get that
\begin{equation}\label{eq002}
  \sum_{B\subset [p], |B|\le r-1}|\ff(B)|+|\ff(B\cup\{p+1\})|\le M\sum_{B'\in{[p+1]\choose r}\setminus {[p]\choose r}}|\ff(B')|,
\end{equation}
where $$M:=(s+1) \sum_{j=0}^{r-1}\frac{s^{r-1-j}{p\choose j}}{{p\choose r-1}}.$$
The following lemma (in the particular case $u=s+1$) was proved in \cite{F4} (see \cite[Theorem~3.1]{F4} and also \cite[Lemma~5]{FK7}).
\begin{lem}\label{lemcd}
  Let $N\ge (u+s)l$ for some $u\in\mathbb Z$, $u\ge s+1$, and suppose that $\mathcal G_1,\ldots, \mathcal G_{s+1}\subset{[N]\choose l}$ are cross-dependent and nested. Then
\begin{equation}\label{eq112} |\mathcal G_1|+|\mathcal G_2|+\ldots +|\mathcal G_{s}|+u|\mathcal G_{s+1}|\le s{N\choose l}.\end{equation}
\end{lem}
Fix $B\in{[p]\choose r-1}$ and assume that $[p]\setminus B = \{j_1,\ldots, j_{p-r+1}\}$. For each $z\in [p-r+1]$, the $(s+1)$-tuple $$\mathcal T(B,z):=\big(\ff(\{j_1\}\cup B),\ldots, \ff(\{j_{p-r+1}\}\cup B),\ff(\{p+1\}\cup B), \ff(\{j_z\}\cup B),\ldots,\ff(\{j_z\}\cup B)\big),$$
(note that $\ff(\{j_z\}\cup B)$ occurs $s-p+r$ times in this tuple in total) is cross-dependent by Claim~\ref{cla1}. Moreover, after reordering, it is nested with the smallest family being $\ff(\{p+1\}\cup B)$.  Apply Lemma~\ref{lemcd} to $\mathcal T(B,z)$ for each $z\in [p-r+1]$ and sum up the corresponding inequalities. Note that $\ff(\{j_z\}\cup B)$ gets coefficient $s$ in this summation ($s-p+r$ from $\mathcal T(B,z)$ and $1$ from each of the  $\mathcal T(B,z')$, $z'\in[p-r+1]\setminus \{z\}$). Thus, for fixed $B\in {[p]\choose r-1}$ and provided \begin{equation}\label{eq121} n\ge p+1+(s+u)(k-r),\end{equation} we get
$$ s\sum_{i\in [p]\setminus B}|\ff(\{i\}\cup B)|+(p-r+1)u|\ff(\{p+1\}\cup B)|\le s(p-r+1){n-p-1\choose k-r}.$$
Next, we sum this inequality over all $B\in {[p]\choose r-1}$. Note that, for each $B'\subset {[p]\choose r}$, it will appear in $r$ summands as above, and, for each $B'\in {[p+1]\choose r}\setminus {[p]\choose r}$, it will appear in exactly $1$ summand. That is, we get
\begin{multline*}sr\sum_{B'\in {[p]\choose r}}|\ff(B')|+(p-r+1)
u\sum_{B'\in{[p+1]\choose r}\setminus {[p]\choose r}}|\ff(B')|\le\\ s(p-r+1){p\choose r-1}{n-p-1\choose k-r}=sr{p\choose r}{n-p-1\choose k-r}.\end{multline*}
Put $\aaa'_{p,r}:= \{A\in \aaa_{p,r}: |A\cap [p+1]|\le r\}$. Note that $\big|\aaa'_{p,r}\cap\{F:|F\cap [p+1]|=r\}\big| = {p\choose r}{n-p-1\choose k-r}$, that is, up to a multiplicative factor $sr$, the right hand side of the last formula above. Therefore, we divide both parts by $sr$ and rewrite this formula as follows.
\begin{equation}\label{eq120}\sum_{B'\in {[p]\choose r}}|\ff(B')|+u'\sum_{B'\in{[p+1]\choose r}\setminus {[p]\choose r}}|\ff(B')|\le |\aaa'_{p,r}|,\end{equation}
where
\begin{equation}\label{eq123}u' = \frac {(p-r+1)
u}{sr}.
\end{equation}
Therefore, if $u'\ge M$, then, using \eqref{eq002}, the inequality \eqref{eq120} implies that
\begin{equation}\sum_{B'\subset [p+1], |B'|\le r}|\ff(B')|\le |\aaa'_{p,r}|,\end{equation}
which, together with \eqref{eq001}, completes the proof of the theorem. Finally,  the inequality $u'\ge M$ is equivalent to
$$u\ge \frac {sr}{(p-r+1)
}\cdot (s+1) \sum_{j=0}^{r-1}\frac{s^{r-1-j}{p\choose j}}{{p\choose r-1}}=s(s+1)\sum_{j=0}^{r-1}\frac{s^{r-1-j} {p\choose j}}{{p\choose r}}=f(s,p,r).$$
Since $n\ge p+1+(s+f(s,p,r))(k-r)$, we may take $u=f(s,p,r)$. Then the inequality above, as well as the inequality \eqref{eq121} is satisfied. This completes the proof.

\section{Final remarks}
The question of determining $m(n,k,s,q)$ in general seems to be very hard since it includes some difficult questions, notably the Erd\H os Matching Conjecture, as a subcase. However, such a generalization of the problem might be easier to deal with by means of induction.

We believe that the first natural case to settle completely is the $k=3$ case. The first author proved the Erd\H os Matching Conjecture for $k=3$ and any $n$ in \cite{F11}. We have obtained some partial results in Theorems~\ref{thmsmallq},~\ref{thmf3} and~\ref{thmf2}, which notably show that each of the $3$ potential extremal families suggested by Conjecture~\ref{conj2} are extremal in some ranges. However, a full resolution requires much more work. It seems that the case of large $s$ may be simpler, in particular because some tools are available for large $s$ (cf. \cite{FK16}). Concluding, we suggest the following particular case of Conjecture~\ref{conj2}.
\begin{prb}
  Determine $m(n,3,s,q)$ for all $s\ge s_0$ and all meaningful $n,q$.
\end{prb}
\section{Acknowledgements}
We thank Adam Zsolt Wagner for pointing out a typo in the formulation of Conjecture~\ref{conj2} and anonymous referees for pointing out several issues with the text.  The research of the second author was supported by the Advanced Postdoc.Mobility grant no. P300P2\_177839 of the Swiss National Science Foundation (results in Section 2). The research of both authors was supported by the Ministry of Education and Science of the Russian Federation in the framework of MegaGrant no 075-15-2019-1926.


\begin{thebibliography}{100}
\bibitem{AK2} R. Ahlswede, L. Khachatrian, \textit{The complete intersection theorem for systems
of finite sets}, Eur. J. Comb. 18 (1997), N2, 125--136.

\bibitem{AF} J. Akiyama and P. Frankl, {\it On the size of graphs with complete-factors,} Journal of Graph Theory 9 (1985), N1, 197--201.

\bibitem{BDE} B. Bollob\'as, D.E. Daykin and P. Erd\H os, \textit{Sets of independent edges of a hypergraph}, Quart. J. Math. Oxford Ser. 27 (1976), N2, 25--32.

\bibitem{E} P. Erd\H os, \textit{A problem on independent r-tuples}, Ann. Univ. Sci. Budapest. 8 (1965) 93--95.

\bibitem{EG} P. Erd\H os and T. Gallai, {\it On the minimal number of vertices representing the edges of a graph}, Puhl. Math. Inst. Hungar. Acnd. Sci. 6 (1961), 181--203.

\bibitem{EKR} P. Erd\H os, C. Ko and R. Rado, \textit{Intersection theorems for systems of finite sets}, The Quarterly Journal of Mathematics, 12 (1961) N1, 313--320.
\bibitem{F76} P. Frankl, {\it Families of finite sets satisfying union restrictions}, Studia Sci. Math. Hungar. 11 (1976) N1-2, 1-6.
\bibitem{F1} P. Frankl, \textit{The Erd\H os--Ko--Rado theorem is true for n=ckt}, Combinatorics (Proc. Fifth Hungarian Colloq., Keszthely, 1976), Vol. I, 365--375, Colloq. Math. Soc. J\'anos Bolyai, 18, North-Holland.

\bibitem{F79} P. Frankl, {\it Families of finite sets satisfying a union condition}, Discrete Math. 26 (1979) N2, 111--118.

\bibitem{F3} P. Frankl, \textit{The shifting technique in extremal set theory}, Surveys in combinatorics, Lond. Math. Soc. Lecture Note Ser. 123 (1987), 81--110, Cambridge University
Press, Cambridge.

\bibitem{F4} P. Frankl, \textit{Improved bounds for Erd\H os' Matching Conjecture}, Journ. of Comb. Theory Ser. A 120 (2013), 1068--1072.

\bibitem{F11} P. Frankl, \textit{On the maximum number of edges in a hypergraph with given matching number}, Disc. Appl. Math. 216 (2017), 562--581.

\bibitem{FF} P. Frankl, Z. F\"uredi, \textit{Beyond the Erd\H os--Ko--Rado theorem}, Journal of Combinatorial Theory, Ser. A 56 (1991) N2, 182--194.

\bibitem{FK7} P. Frankl and A. Kupavskii, {\it Two problems on matchings in set families -- in the footsteps of Erd\H os and Kleitman}, to appear in J. Comb. Th. Ser. B., arXiv:1607.06126

\bibitem{FK16} P. Frankl and A. Kupavskii, {\it  The Erd\H os Matching Conjecture and Concentration Inequalities,} arXiv:1806.08855
\bibitem{HLS} H. Huang, P.-S. Loh  and B. Sudakov, {\it The Size of a Hypergraph and its Matching Number}, Combinatorics, Probability and Computing 21 (2012), N3, 442--450.
\bibitem{Ka1964} G.O.H. Katona, {\it Intersection theorems for systems of finite sets}, Acta Mathematica Hungarica, 15(1964), N3-4, 329--337.
\end{thebibliography}
\end{document}